\documentclass[letterpaper,twoside,11pt,leqno]{article}
\usepackage{amsthm}
\usepackage{amsmath}
\usepackage{natbib}
\RequirePackage[colorlinks]{hyperref}
\usepackage{amsfonts}

\def\authors{George Lowther}
\def\runauthor{\authors}
\def\head{Properties of Expectations of Functions of Martingale Diffusions}
\def\runhead{Expectations of Diffusions}
\def\keywords{diffusion, martingale, strong Markov property.}
\def\classifications{60J60,60J25, 60G44.}

\makeatletter
\def\@evenhead{\thepage\hfill{\small\MakeUppercase{\runauthor}}\hfill}
\def\@oddhead{\hfill{\small\MakeUppercase{\runhead}}\hfill\thepage}
\def\blfootnote{\xdef\@thefnmark{}\@footnotetext}
\makeatother

\newcommand{\halfplane}{\reals_+\times\reals}
\newcommand{\reals}{\mathbb{R}} 
\newcommand{\rats}{\mathbb{Q}} 
\newcommand{\E}[1]{\mathbb{E}\left[#1\right]} 
\newcommand{\setsF}{\mathcal{F}} 

\newcommand{\pd}[1]{_{\!,_{#1}}}

\newcommand{\nat}{\mathbb{N}}
\newcommand{\Prob}[1]{\mathbb{P}\left(#1\right)} 
\newcommand{\acd}{ACD}
\newcommand{\support}[1]{{\rm Supp}(#1)}
\newcommand{\msupport}[1]{{\rm MSupp}(#1)}
\newcommand{\slbrack}{{[\hspace*{-0.15em}[}}  
\newcommand{\srbrack}{{]\hspace*{-0.15em}]}}  
\newcommand{\PP}{\mathbb{P}}
\newcommand{\filtrationF}{(\Omega,\setsF, (\setsF_t)_{t\in\reals_+},\PP)}

\newenvironment{eqnalign*}{\setlength\arraycolsep{2pt}\begin{eqnarray*}}{\end{eqnarray*}}

\newtheorem{definition}{Definition}[section]
\newtheorem{theorem}[definition]{Theorem}
\newtheorem{lemma}[definition]{Lemma}
\newtheorem{corollary}[definition]{Corollary}

\begin{document}

\title{\head}
\blfootnote{\emph{Key Words}: \keywords}
\blfootnote{\emph{AMS 2000 Classification}: \classifications}
\author{\authors}
\date{}
\maketitle
\thispagestyle{empty}

\begin{abstract}
Given a real valued and time-inhomogeneous martingale diffusion $X$, we investigate the properties of functions defined by the conditional expectation $f(t,X_t)=\E{g(X_T)|\setsF_t}$.
We show that whenever $g$ is monotonic or Lipschitz continuous then $f(t,x)$ will also be monotonic or Lipschitz continuous in $x$. If $g$ is convex then $f(t,x)$ will be convex in $x$ and decreasing in $t$.
We also define the \emph{marginal support} of a process and show that it almost surely contains the paths of the process.
Although $f$ need not be jointly continuous, we show that it will be continuous on the marginal support of $X$. We prove these results for a generalization of diffusion processes that we call `almost-continuous diffusions', and includes all continuous and strong Markov processes.
\end{abstract}

\section{Introduction}
\label{sec:intro}

Suppose that we have a real valued Markov process $X$, and any times $T>t\ge 0$. Then, the Markov property says that for any measurable function $g:\reals\rightarrow\reals$ such that $g(X_T)$ is integrable, there is a function $f$ satisfying
\begin{equation}\label{eqn:cond exp}
f(t,X_t) = \E{g(X_T)|\setsF_t}
\end{equation}
for every $t\le T$. The aim of this paper is to show that various properties of $f$ follow from the properties of $g$. When $X$ is a time inhomogeneous diffusion (i.e., strong Markov and continuous) then it is shown that $f(t,x)$ is increasing in $x$ whenever $g$ is increasing. If $X$ is also a martingale we show that $f(t,x)$ is Lipschitz continuous in $x$ whenever $g(x)$ is Lipschitz continuous.
If, furthermore, $g$ is convex then $f(t,x)$ will be both convex in $x$ and decreasing in $t$.
The proof of these results makes use of the coupling methods used by \citep{Hobson}.

We also look at the continuity of $f$. In general, it may not be possible to choose $f(t,x)$ to be  jointly continuous everywhere. However, we define the \emph{marginal support} of a process, which is a Borel subset of $\halfplane$ defined in terms of the supports of the marginal distributions, and show that it almost-surely contains the paths of the process and that $f$ is jointly continuous on this set.

All of these results are proven  for a generalization of diffusion processes that we refer to as \emph{almost-continuous} diffusions. Such processes are defined by the strong Markov property and a pathwise almost-continuity condition, and we also characterize the almost-continuous martingales in terms of the jumps of the process and its marginal support.

The situation where $g$ is convex and monotonic occurs frequently in mathematical finance.
In this case, $g(X_T)=(X_T-k)_+$ is the payoff of a call option with strike $k$ and maturity $T$, and $f(t,X_t)$ is its value at time $t$ as a function of the underlying asset price $X$. It is well known that under many situations the function $f(t,x)$ would then be convex and increasing in $x$ and decreasing in $t$. For example, this is a property of the Black-Scholes pricing formula, or indeed any model in which $X$ is the exponential of a process with independent increments (see, for example, \citep{Merton73}).

Another case that has been studied by several authors is when $X$ is a diffusion described by an SDE of the form
\begin{equation}\label{eqn:SDE}
dX_t=\sigma(t,X_t)\,dW_t
\end{equation}
where $W$ is a Brownian motion. This situation has been investigated by \citep{Bergman,Hobson,Karoui}. In \citep{Bergman} they differentiate the following PDE for $f$,
\begin{equation*}\label{eqn:PDE}
\frac{\partial f}{\partial t}(t,x) + \frac{1}{2}\sigma^2(t,x)\frac{\partial^2 f}{\partial x^2}(t,x) = 0.
\end{equation*}
Then it is shown that the derivative of $f$ is increasing with respect to $x$. In \citep{Karoui} they use a method of stochastic flows. Both these approaches require $\sigma$ to be continuously differentiable. In \citep{Janson} the notion of volatility time is used to show that $f(t,x)$ is decreasing in $t$, and then deduce convexity in $x$ from the PDE above.
In \citep{Hobson} a simple coupling method is used to prove convexity in $x$, which doesn't place any restrictions on $\sigma$, other than those necessary for the SDE (\ref{eqn:SDE}) to have unique solutions that are martingales.
We employ these coupling techniques in the current paper in order to prove the monotonicity, Lipschitz continuity and convexity results for $f$. However, rather than assuming that $X$ is the unique solution of an SDE such as (\ref{eqn:SDE}) we instead work with the more general definition of a diffusion as a continuous and  strong Markov process, and further generalize to the following class of processes.
Except where explicitly stated otherwise, we assume that all stochastic processes are defined with respect to a filtered probability space $\filtrationF$.

\begin{definition}\label{defn:acd}
Let $X$ be a real valued stochastic process. Then,
\begin{enumerate}
\item $X$ is \emph{strong Markov} if for every bounded, measurable $g:\reals\rightarrow\reals$ and every $t\in\reals_+$ there exists a measurable $f:\halfplane\rightarrow\reals$ such that
\begin{displaymath}
f(\tau,X_\tau) = \E{g(X_{\tau+t})\mid \setsF_\tau}
\end{displaymath}
for every finite stopping time $\tau$.

\item $X$ is almost-continuous if it is cadlag, continuous in probability and given any two independent, identically distributed cadlag processes $Y,Z$ with the same distribution as $X$ and for every $s<t\in\reals_+$ we have
\begin{displaymath}
\Prob{Y_s<Z_s, Y_t>Z_t{\rm\ and\ }Y_u\not= Z_u\textrm{ for every }u\in(s,t)}=0
\end{displaymath}

\item $X$ is an \emph{almost-continuous diffusion} if it is strong Markov and almost-continuous.
\end{enumerate}
\end{definition}

The second condition above is equivalent to saying that $Y-Z$ cannot change sign without hitting zero and, by the intermediate value theorem, is clearly true for continuous processes. Examples of almost-continuous diffusions that are not continuous include, for example, Poisson and compensated Poisson processes.
Also, almost-continuous diffusions arise as limits of continuous diffusions under the topology of convergence of finite-dimensional distributions (as we shall show in a future paper).
We will often abbreviate `almost-continuous diffusion' to \acd.

To further see what the almost-continuous property means, we can look at an alternative way of defining it. First, given any real valued random variable $X$ then its support, $\support{X}$, is defined to be the smallest closed subset of $\reals$ such that $\Prob{X\in\support{X}}=1$. We define the \emph{marginal support} of a process $X$ in terms of the supports of its marginal distributions.

\begin{definition}\label{defn:marginal support}
Let $X$ be a real valued stochastic process. Then, its \emph{marginal support} is
\begin{displaymath}
\msupport{X} = \left\{(t,x)\in\halfplane: x \in\support{X_t}\right\}.
\end{displaymath}
\end{definition}

If $X$ is right-continuous in probability then, as we shall show, its marginal support will be a Borel measurable subset of $\halfplane$. We now state the alternative definition of almost-continuous martingales.

\begin{lemma}\label{lemma:mgale AC iff never jumps past msupp}
Let $X$ be a real valued cadlag martingale which is continuous in probability. Then it is almost-continuous if and only if
\begin{equation}
\label{eqn:lemma:mgale AC iff never jumps past msupp:1}\left\{(t, x)\in\halfplane : X_{t-} < x < X_t\textrm{ or }X_t < x < X_{t-}\right\}
\end{equation}
is almost surely disjoint from the marginal support of $X$.
\end{lemma}

So a martingale is almost-continuous if it cannot jump past any points in its marginal support. The proof of this is rather complicated, and is left until Section \ref{sec:msupport}.

We now state the result that $f(t,x)$ satisfying (\ref{eqn:cond exp}) can be chosen to be increasing in $x$ whenever $g$ is increasing.

\begin{theorem}\label{thm:exp of increasing is increasing}
Let $X$ be an almost-continuous diffusion. Choose any increasing function $g:\reals\rightarrow\reals$ and $T\in\reals_+$ such that $\E{|g(X_T)|}<\infty$.

Then, for every $t\in[0,T]$ there exists an increasing function $f:\reals\rightarrow\reals\cup\{\pm\infty\}$ such that
$f(X_t)=\E{g(X_T)|\setsF_t}$.
\end{theorem}

The proof of this is left until Section \ref{sec:expectations}.
Note that this applies to all almost-continuous diffusions, and not just martingales. Under the additional assumption that $X$ is a martingale then by a simple corollary of Theorem \ref{lemma:mgale AC iff never jumps past msupp}, $f$ will be Lipschitz continuous in $x$ whenever $g$ is.
Recall that a function $f$ on the real numbers is Lipschitz continuous if there exists real numbers $k\le K$ such that
\begin{equation}\label{eqn:exp:lipschitz property}
k(y-x)\le f(y)-f(x)\le K(y-x)
\end {equation}
for all $x<y\in\reals$. This is satisfied if and only if the derivative $df(x)/dx$ exists in the measure theoretic sense and can be chosen such that $k\le df(x)/dx\le K$ for all $x\in\reals$.
So, we shall write $k\le f^\prime\le K$ to mean that $f$ satisfies inequality (\ref{eqn:exp:lipschitz property}) for all $x<y\in\reals$.

\begin{theorem}\label{thm:exp of lip is lip}
Let $X$ be an \acd\ martingale, and $g:\reals\rightarrow\reals$ be a Lipschitz continuous function such that $k\le g^\prime\le K$. Then for any $t<T\in\reals_+$ there exists a Lipschitz continuous $f:\reals\rightarrow\reals$ such that $k\le f^\prime\le K$ and
$f(X_t)=\E{g(X_T)|\setsF_t}$.
\end{theorem}
\begin{proof}
As $X$ is a martingale and $k\le g^\prime\le K$,
\begin{equation*}
\E{|g(X_T)|}\le |g(0)|+\max{(-k,K)}\E{|X_T|}<\infty,
\end{equation*}
As $X$ is also Markov, we can define $f$ by $f(X_t)=\E{g(X_T)|\setsF_t}$. Then, the martingale property gives
\begin{eqnalign*}
&&f(X_t)-kX_t=\E{g(X_T)-kX_T|\setsF_t},\\
&&KX_t-f(X_t)=\E{KX_T-g(X_T)|\setsF_t}.
\end{eqnalign*}
As $k\le g^\prime\le K$ we see that $g(x)-kx$ and $Kx-g(x)$ are increasing, so Theorem \ref{thm:exp of increasing is increasing} tells us that $f(x)-kx$ and $Kx-f(x)$ are increasing on some set
$S\subseteq\reals$ with $\Prob{X_t\in S}=1$. Therefore equation (\ref{eqn:exp:lipschitz property}) is satisfied for all $x<y\in S$. Finally, $f$ extends to the closure of $S$ by uniform continuity, and can be linearly interpolated and extrapolated outside of the support of $X_t$.
\end{proof}

We now state the properties of $f(t,x)$ satisfying (\ref{eqn:cond exp}) when $g$ is convex. This is the case investigated in \citep{Bergman,Hobson,Karoui} and shows that whenever the price of a financial asset is modeled by a continuous (or almost-continuous) and strong Markov process then option prices will be convex in the asset price and decreasing in time.

\begin{theorem}\label{thm:exp of conv is conv and dec in t}
Let $X$ be an \acd\ martingale, $g:\reals\rightarrow\reals$ be a convex and Lipschitz continuous function and choose any $T\in\reals_+$.
Then there exists a function $f:[0,T]\times\reals\rightarrow\reals$ such that $f(t,x)$ is convex in $x$, decreasing in $t$ and 
\begin{equation}\label{eqn:cond exp of conv}
f(t,X_t)=\E{g(X_T)|\setsF_t}
\end{equation}
for every $t\in[0,T]$.
\end{theorem}

The proof of this is given in Section \ref{sec:expectations}, and follows the coupling idea of \citep{Hobson}. As we now show, $f$ will be jointly continuous on the marginal support of $X$. The requirement in the following result that locally $f(t,x)$ is uniformly continuous in $x$, just means that it is uniformly continuous in $x$ on every bounded subset of $[0,T]\times\reals$.

\begin{theorem}
Let $X$ be a real valued and adapted process which is continuous in probability. Suppose that $T>0$ and $f:[0,T]\times\reals\rightarrow\reals$ is such that locally $f(t,x)$ is uniformly continuous in $x$ and $f(t,X_t)$ is a martingale. Then, $f$ is jointly continuous on the marginal support of $X$.
\end{theorem}
\begin{proof}
Set $M_t=f(t,X_t)$ for $t\le T$. As it is a martingale, it has left and right limits everywhere. We shall prove left and right continuity separately. First choose any sequence $t_n\in[0,T]$ with $t_n\downarrow\downarrow t$. The continuity in probability of $X_t$ and the local uniform continuity of $f(t,x)$ in $x$ gives
\begin{equation*}
M_{t_+}=\lim_{n\rightarrow\infty} f(t_n,X_{t_n})=\lim_{n\rightarrow\infty}f(t_n,X_t)
\end{equation*}
where convergence is in probability. In particular, we see that $M_{t+}$ is $\setsF_t$-measurable, so $M_t=\E{M_{t+}|\setsF_t}=M_{t+}$ and
\begin{equation*}
f(t,X_t)=M_t=\lim_{n\rightarrow\infty}f(t_n,X_t).
\end{equation*}
Then the local uniform continuity of $f(t,x)$ in $x$ shows that $f(t_n,x)\rightarrow f(t,x)$ for all $x$ in the support of $X_t$, and $f(t,x)$ is indeed right-continuous in $t$ on $\msupport{X}$.

Now choose $t_n\in[0,T]$ such that $t_n\uparrow\uparrow t$. Arguing as above, we have
\begin{equation*}
M_{t-}=\lim_{n\rightarrow\infty}f(t_n,X_{t_n})=\lim_{n\rightarrow\infty}f(t_n,X_t).
\end{equation*}
However, the continuity in probability of $X_t$ implies that $X_t$ is $\setsF_{t-}$-measurable. So, $M_{t-}=\E{M_t|\setsF_{t-}}=M_t$ and,
\begin{equation*}
f(t,X_t)=M_t=\lim_{n\rightarrow\infty}f(t_n,X_t).
\end{equation*}
As above, this implies that $f(t_n,x)\rightarrow f(t,x)$ for every $x$ in the support of $X_t$, and $f(t,x)$ is left-continuous in $t$ on the marginal support of $X$.
\end{proof}

In particular, suppose that $X$ is an \acd\ martingale, $g$ is Lipschitz continuous, and $f$ is defined by equation (\ref{eqn:cond exp}). Then $f(t,X_t)$ is a martingale and Theorem \ref{thm:exp of lip is lip} says that $f(t,x)$ can be chosen to be Lipschitz continuous in $x$. The above result then states that $f$ is jointly continuous on $\msupport{X}$.
As we shall show in Section \ref{sec:msupport}, the paths of $X$ will lie within its marginal support (Lemmas \ref{lemma:process lies in msupp} and \ref{lemma:uniqueness:left lim of x is in msupp}) at all times, so we can conclude that $f(t,X_t)$ almost surely has cadlag paths and can only jump at times when $X$ jumps.

In fact, as we now show by a counterexample, it is generally not possible to choose $f$ to be continuous outside of the marginal support of $X$.
To construct our example, we first let $B$ be a Brownian motion and let $\tau$ be the random time
\begin{equation*}
\tau=\inf\left\{t\in\reals_+:|B_t|=1\right\}<\infty.
\end{equation*}
This is a stopping time with respect to the filtration generated by $B$, and $B_\tau=\pm 1$. Define the process $X$ by
\begin{equation*}
X_t=\left\{
\begin{array}{ll}
B_{t/(1-t)},&\textrm{if }0\le t<\tau/(1+\tau),\\
B_\tau,&\textrm{if }\tau/(1+\tau)\le t\le 1,\\
B_{\tau+t-1},&\textrm{if }t>1.
\end{array}
\right.
\end{equation*}
Then, under its natural filtration, $X$ is a continuous strong Markov martingale satisfying the SDE (\ref{eqn:SDE}) with
\begin{equation*}
\sigma(t,x) =\frac{1}{1-t}1_{\{(t,x)\in[0,1)\times(-1,1)\}} + 1_{\{t\ge 1\}}
\end{equation*}
The marginal support of $X$ is
\begin{equation*}
\msupport{X}=\left\{(0,0),(1,-1),(1,1)\right\}\cup\left((0,1)\times[-1,1]\right)\cup\left((1,\infty)\times\reals\right).
\end{equation*}
Note, in particular, that at time $t=1$ the support of $X_t$ is not connected, consisting of just the two points $\{-1,1\}$.
Now let $g(x)=x^2$, choose any $T>1$ and let $f(t,x)$ satisfy equation (\ref{eqn:cond exp}). If $t\in[1,T]$ then
\begin{equation*}
f(t,X_t)=\E{X_T^2|\setsF_t}=X_t^2 + T-t.
\end{equation*}
And, for $t\le 1$ we can use $X_1=\pm 1$ to get
\begin{equation*}
f(t,X_t)=\E{f(1,X_1)|\setsF_t}=\E{X_1^2+T-1|\setsF_t}=T
\end{equation*}
As $|X_t|\le 1$ for $t< 1$, we can choose whatever value we like for $f(t,x)$ in the range $t<1$ and $|x|>1$. In particular, $f$ can be chosen to be
\begin{equation*}
f(t,x) = \left\{
\begin{array}{ll}
x^2+T-t,&\textrm{if $1\le t\le T$},\\
T,&\textrm{if $t<1$ and $|x|\le 1$},\\
x^2+T-1,&\textrm{if $t<1$ and $|x|>1$}.
\end{array}
\right.
\end{equation*}
Note that this is convex in $x$ and right-continuous and decreasing in $t$. Although Theorem \ref{thm:exp of conv is conv and dec in t} does not quite apply in this case, because the derivative of $g$ is not bounded, it should be clear that it could be approximated by Lipschitz continuous functions, so the fact that $f(t,x)$ is convex in $x$ and decreasing in $t$ will still follow from Theorem \ref{thm:exp of conv is conv and dec in t}. We just chose $g(x)=x^2$ for simplicity, but any convex and non-linear function would do just as well.

We see that $f$ is jointly continuous everywhere except for the line $\{1\}\times(-1,1)$, on which
\begin{eqnalign*}
&&f(1,x) = x^2 + T-1,\\
&&f(1-,x) = T.
\end{eqnalign*}
So $f$ is discontinuous on this line, but is continuous everywhere on the marginal support of $X$.

\section{The Strong Markov Property}

In order to prove the results of Section \ref{sec:intro} we shall make use of the strong Markov property. However, the form in which it is stated in Definition \ref{defn:acd} is not the easiest to work with, so we shall make use of the following lemma which rephrases the strong Markov property in a slightly different way.
We are only interested in real valued processes here, although the following proofs will generalize to any Polish space.

\begin{lemma}\label{lemma:strong markov:alternative defn}
Let $X$ be a real valued cadlag process that satisfies the strong Markov property.
For every $T\in\reals_+$ and bounded measurable $g:\reals\rightarrow\reals$ there exists a measurable $f:[0,T]\times\reals\rightarrow\reals$ such that for every stopping time $\tau$
\begin{equation*}
1_{\{\tau\le T\}}f(\tau,X_\tau) = 1_{\{\tau\le T\}}\E{g(X_T)\mid \setsF_\tau}.
\end{equation*}
\end{lemma}
\begin{proof}
Choose any $T\in\reals_+$ and bounded continuous $g:\reals\rightarrow\reals$. From Definition \ref{defn:acd} of the strong Markov property, for every $\alpha>0$ there exists a measurable $f_\alpha:\halfplane\rightarrow\reals$ such that for every stopping time $\tau$,
\begin{equation*}
1_{\{T<\infty\}}f_\alpha(\tau,X_\tau) = 1_{\{\tau<\infty\}}\E{g(X_{\tau+\alpha})\mid \setsF_\tau}.
\end{equation*}
Now pick any $n\in\nat$. For any stopping time $\tau$ let $\tau_n$ be the stopping time
\begin{equation*}
\tau_n=\left\{
\begin{array}{ll}
\infty,&\textrm{if }\tau>T,\\
\tau + ([n(T-\tau)]+1)/n,&\textrm{if }\tau\le T.
\end{array}
\right.
\end{equation*}
Here we are using the notation $[x]$ to denote the largest integer that is less than or equal to $x$. Define the measurable function $h_n:[0,T]\times\reals\rightarrow\reals$ by
\begin{equation*}
h_n(t,x)=\sum_{k=1}^\infty1_{\left\{T-\frac{k}{n}<t\le T-\frac{k-1}{n}\right\}}f_{\frac{k}{n}}(t,x).
\end{equation*}
So,
\begin{eqnalign*}
1_{\{\tau\le T\}}h_n(\tau,X_\tau) &=& \sum_{k=1}^\infty1_{\left\{t-\frac{k}{n}<\tau\le T-\frac{k-1}{n}\right\}} \E{g\left(X_{\tau+\frac{k}{n}}\right)\big|\setsF_\tau}\\
&=&\E{\sum_{k=1}^\infty1_{\left\{T-\frac{k}{n}<\tau\le T-\frac{k-1}{n}\right\}} g\left(X_{\tau+\frac{k}{n}}\right)\bigg|\setsF_\tau}\\
&=&1_{\{\tau\le T\}}\E{g(X_{\tau_n})|\setsF_\tau}.
\end{eqnalign*}
From the definition of $\tau_n$ we have $T<\tau_n\le T+\frac{1}{n}$ whenever $\tau\le T$, so $\tau_n\rightarrow T$ as $n\rightarrow\infty$.
By the right-continuity of $X$ and the continuity of $g$, we can use bounded convergence to get
\begin{equation*}
1_{\{\tau\le T\}}\E{g(X_{T})|\setsF_\tau}=\lim_{n\rightarrow\infty}1_{\{\tau\le T\}}h_n(\tau,X_\tau).
\end{equation*}
So the result for continuous $g$ follows by setting
\begin{equation*}
f(t,x)=\limsup_{n\rightarrow\infty}h_n(t,x).
\end{equation*}
Then it extends to arbitrary bounded measurable $g$ by the Monotone Class Lemma.
\end{proof}

Given a strong Markov process $X$ we shall require the existence of a process which is independent of $X$ and with the same distribution.
One way to construct such a process is to take the product of the underlying probability space with itself.

For any two filtered probability spaces
\begin{equation*}
\left(\Omega_1,\setsF^1,(\setsF^1_t)_{t\in\reals_+},\PP_1\right),\ 
\left(\Omega_2,\setsF^2,(\setsF^2_t)_{t\in\reals_+},\PP_2\right)
\end{equation*}
we can form the product
\begin{equation}\label{eqn:strongmarkov:product of bases}
\left(\Omega,\setsF,(\setsF_t)_{t\in\reals_+},\PP\right)
=
\left(\Omega_1\times\Omega_2,\setsF^1\otimes\setsF^2,(\setsF^1_t\otimes\setsF^2_t)_{t\in\reals_+},\PP_1\otimes\PP_2\right).
\end{equation}
Any process defined on one of the two original probability spaces lifts to a process defined on their product (by combining with the projection map). When we pass to a larger probability space like this, we need to know that strong Markov processes will still be strong Markov.

\begin{lemma}\label{lemma:strong markov stays sm in extension}
Let
\begin{eqnalign*}
&&X:\reals_+\times\Omega_1\rightarrow\reals,\\
&&(t,\omega_1)\mapsto X_t(\omega_1)
\end{eqnalign*}
be a strong Markov process on the filtered probability space $\left(\Omega_1,\setsF^1,(\setsF^1_t)_{t\in\reals_+},\PP_1\right)$.

If we lift $X$ to the process $\tilde X$ defined on the product of the filtered probability spaces in equation (\ref{eqn:strongmarkov:product of bases})
\begin{eqnalign*}
&&\tilde X:\reals_+\times\Omega\rightarrow\reals,\\
&&(t,\omega_1,\omega_2)\mapsto \tilde X_t(\omega_1,\omega_2)\equiv X_t(\omega_1)
\end{eqnalign*}
then $\tilde X$ is also strong Markov.
\end{lemma}
\begin{proof}
Choose any bounded measurable $g:\reals\rightarrow\reals$ and $\alpha>0$. As $X$ is strong Markov there exists a measurable $f:\halfplane\rightarrow\reals$ such that for every finite stopping time $\tau$ defined on the first filtered probability space,
\begin{equation*}
1_{\{\tau<\infty\}}f(\tau,X_\tau)=1_{\{\tau<\infty\}}\E{g(X_{\tau+\alpha})|\setsF^1_\tau}.
\end{equation*}
This is equivalent to the statement
\begin{equation*}
\E{U\left(f(\tau,X_\tau)-g(X_{\tau+\alpha})\right)}=0
\end{equation*}
for every bounded $\setsF^1_\tau$-measurable random variable $U$.

Now choose any finite stopping time $\tau$ defined with respect to the product of the filtered probability spaces, and let $U$ be a bounded $\setsF_\tau$-measurable random variable. For any $\omega_2\in\Omega_2$ define the $\setsF^1$-stopping time
\begin{eqnalign*}
&&\tau_{\omega_2}:\Omega_1\rightarrow\reals_+,\\
&&\tau_{\omega_2}(\omega_1)= \tau(\omega_1,\omega_2)
\end{eqnalign*}
and the $\setsF^1_{\tau_{\omega_2}}$-measurable random variable
\begin{eqnalign*}
&&U_{\omega_2}:\Omega_1\rightarrow\reals,\\
&&U_{\omega_2}(\omega_1)=U(\omega_1,\omega_2).
\end{eqnalign*}
Then,
\begin{eqnalign*}
&&\E{U\left(g(\tau,\tilde X_\tau)-f(\tilde X_{\tau+\alpha})\right)}\\
&&=\int\E{U_{\omega_2}\left(g(\tau_{\omega_2},X_{\tau_{\omega_2}}) - f(X_{\tau_{\omega_2}+\alpha})\right)}\,d\PP(\omega_2)\\
&&=0
\end{eqnalign*}
so $\tilde X$ is indeed strong Markov.
\end{proof}

\section{Conditional Expectations}
\label{sec:expectations}

We now move on to proving theorems \ref{thm:exp of increasing is increasing} and \ref{thm:exp of conv is conv and dec in t}.
The method of proof employs a similar idea to the coupling arguments used in \citep{Hobson}. The idea is to take several independent copies of the process, and observe them up until the first time that they touch.

If $X$ is any strong Markov process, $T\in\reals_+$ and $g:\reals\rightarrow\reals$ is any measurable function satisfying $\E{|g(X_T)|}<\infty$, then it follows from Lemma \ref{lemma:strong markov:alternative defn} that there exists a measurable $h:[0,T]\times\reals\rightarrow\reals$ such that
$h(T,x)=g(x)$ and
\begin{equation}\label{eqn:exp:increasing:strong markov coupling}
1_{\{\tau>t\}}h(t,X_t)= 1_{\{\tau> t\}}\E{h(T\wedge \tau,X_{T\wedge \tau})|\setsF_t}.
\end{equation}
for any $t\in[0,T]$ and any stopping time $\tau$.
We can make use of this to prove the results we need for conditional expectations of almost-continuous diffusions. We start by proving Theorem \ref{thm:exp of increasing is increasing}, that conditional expectations of increasing functions are themselves increasing.

\begin{proof}[Proof of Theorem \ref{thm:exp of increasing is increasing}]
First we can extend the probability space by taking the product of the underlying filtered probability space with itself. Then there exists a process $Y$ independent of $X$ which is also strong Markov and has the same distribution as $X$ (from Lemma \ref{lemma:strong markov stays sm in extension}). As $X$ is strong Markov there exists a measurable $h:[0,T]\times\reals\rightarrow\reals$ such that $h(T,x)=g(x)$ and equation (\ref{eqn:exp:increasing:strong markov coupling}) is satisfied for any stopping time $\tau$. Setting $f(x)= h(t,x)$ gives $f(X_t)=\E{g(X_T)|\setsF_t}$.
We just need to show that $f$ can be chosen to be increasing.

Let $\tau$ be the stopping time
\begin{equation*}
\tau = \inf\left\{s\in[t,\infty):X_s\ge Y_s\right\}.
\end{equation*}
Note that $\{\tau>t\}=\{X_t<Y_t\}$.

As $X$ is almost-continuous, $X_\tau=Y_\tau$ whenever $t<\tau<\infty$. Also, as $g$ is increasing, $g(X_T)\le g(Y_T)$ whenever $\tau>T$.
So equation (\ref{eqn:exp:increasing:strong markov coupling}) gives
\begin{eqnalign*}
1_{\{\tau>t\}}\left( f(Y_t)-f(X_t)\right)
&=&1_{\{\tau>t\}}\E{h(T\wedge \tau,Y_{T\wedge \tau})-h(T\wedge \tau,X_{T\wedge \tau})\big|\setsF_t}\\
&=&1_{\{\tau>t\}}\E{1_{\{\tau>T\}}\left(g(Y_T)-g(X_T)\right)\big|\setsF_t}
\ge 0.
\end{eqnalign*}
Therefore, $f(X_t)\le f(Y_t)$ whenever $X_t<Y_t$ (a.s.).

This shows that $f$ is increasing in an almost sure sense. That is, there exists a measurable subset $S$ of $\reals$ such that $f$ is increasing on $S$ and
$\Prob{X_t\in S}=1$. So, we can extend $f$ outside of $S$ by
\begin{equation*}
f(x) = \sup\left\{f(y):y\in(-\infty,x]\cap S\right\}\in\reals\cup\{\pm\infty\}.\qedhere
\end{equation*}
\end{proof}

We now show that convexity is also preserved under taking conditional expectations of functions of \acd\ martingales.
This was proven in \citep{Hobson} for the case of continuous diffusions driven by an SDE, by considering running three independent copies of the process and using a coupling method. However, their proof only really relied on continuity and the strong Markov condition, and can be extended to all \acd\ martingales.

We use the condition that a function $f:S\rightarrow\reals$ for $S\subseteq\reals$ is convex if and only if
\begin{equation}\label{eqn:convex:conv condition}
(z-y)f(x)+(y-x)f(z)+(x-z)f(y)\ge 0
\end{equation}
for any $x<y<z\in S$.
Also, for any convex function $f$ from the reals to the reals, the right hand derivative $f^\prime(x)$ exists and is right-continuous and increasing in $x$.

\begin{lemma}\label{lemma:exp:exp of conv is conv}
Let $X$ be an \acd\ martingale, and $g:\reals\rightarrow\reals$ be a convex function with $k\le g^\prime\le K$. Then for any $t<T\in\reals_+$ there exists a convex function $f:\reals\rightarrow\reals$ such that
\begin{equation}\label{eqn:cond exp of convex func}
f(X_t)=\E{g(X_T)|\setsF_t}
\end{equation}
and $k\le f^\prime\le K$.
\end{lemma}
\begin{proof}
As $k\le g^\prime \le K$, $g$ is Lipschitz continuous. By Theorem \ref{thm:exp of lip is lip} there exists a Lipschitz continuous $f:\reals\rightarrow\reals$ with $k\le f^\prime\le K$ and satisfying equation (\ref{eqn:cond exp of convex func}).

We can extend the probability space by taking the product of three copies of the underlying filtered probability space. Then, by Lemma \ref{lemma:strong markov stays sm in extension}, there exists independent strong Markov processes $X^1,X^2,X^3$ each of which have the same distribution as $X$. As the $X^i$ are strong Markov there exists a measurable $h:[0,T]\times\reals\rightarrow\reals$ such that $h(T,x)=g(x)$ and equation (\ref{eqn:exp:increasing:strong markov coupling}) is satisfied for any stopping time $\tau$, and with $X^i$ in place of $X$ ($i=1,2,3$). Then,
\begin{equation*}
f(X_t)=h(t,X_t)\textrm{ (a.s.)}.
\end{equation*}
We now define the process
\begin{equation*}
M_s = (X^3_s-X^2_s)h(s,X^1_s) + (X^2_s-X^1_s)h(s,X^3_s)+(X^1_s-X^3_s)h(s,X^2_s)
\end{equation*}
and the stopping time
\begin{equation*}
\tau = \inf\left\{s\in[t,\infty):X^1_s\ge X^2_s\textrm{ or }X^2_s\ge X^3_s\right\}. \end{equation*}
In particular, note that $\{\tau>t\}=\{X^1_t<X^2_t<X^3_t\}$.

As the $X^i$ are cadlag martingales adapted to independent filtrations, equation (\ref{eqn:exp:increasing:strong markov coupling}) gives
\begin{eqnalign*}
1_{\{\tau>t\}}X^i_t f(X^j_t) &=& 1_{\{\tau>t\}}\E{X^i_{T\wedge \tau}h(T\wedge \tau,X^j_{T\wedge \tau})\big|\setsF_t}
\end{eqnalign*}
whenever $i\not=j$.
So, as $M$ is a linear combination of such terms,
\begin{equation*}
1_{\{\tau>t\}}M_t = 1_{\{\tau>t\}}\E{M_{T\wedge \tau}|\setsF_t}.
\end{equation*}

Now, if $t<\tau\le T$ then either $X^1_\tau=X^2_\tau$ or $X^3_\tau=X^2_\tau$ (because $X$ is almost-continuous). So, $M_\tau=0$.
On the other hand, if $\tau> T$ then $X^1_T<X^2_T<X^3_T$, so the convexity of $g$ gives $M_T\ge 0$.

In either case this says that $M_{T\wedge \tau}\ge 0$ whenever $\tau>t$, so $1_{\{\tau>t\}}M_t\ge 0$ or, equivalently,
\begin{equation*}
(X^3_t-X^2_t)f(X^1_t) + (X^2_t-X^1_t)f(X^3_t)+(X^1_t-X^3_t)f(X^2_t)\ge 0
\end{equation*}
whenever $X^1_t<X^2_t<X^3_t$ (a.s.).
As $f$ is continuous, this shows that it is convex on the support of $X_t$.

We can extend $f$ to any bounded open interval in the complement of $\support{X_t}$ by linear interpolation. If $\support{X_t}$ is uniformly bounded above then we can extrapolate $f$
linearly with gradient $K$ above the support of $X_t$.
Similarly, if $\support{X_t}$ is bounded below then we extrapolate $f$ linearly with gradient $k$ below the support of $X_t$, which gives a convex function.
\end{proof}

We can extend the previous result to prove that conditional expectations of convex functions give convex functions that are decreasing in time, as was stated in Theorem \ref{thm:exp of conv is conv and dec in t}.

\begin{proof}[Proof of Theorem \ref{thm:exp of conv is conv and dec in t}]
As $g$ is Lipschitz continuous, there exist real numbers $k,K$ such that $k\le g^\prime\le K$.
By Lemma \ref{lemma:exp:exp of conv is conv} there is a function $f:[0,T]\times\reals\rightarrow\reals$ satisfying equation (\ref{eqn:cond exp of conv}) and such that $f(t,x)$ is convex and Lipschitz-continuous in $x$ with derivative satisfying $k\le f\pd{2}\le K$.
By linear interpolation and extrapolation outside the support of $X_t$, we can suppose that $f(t,x)$ is linear in $X$ across each of the connected components of $\reals\setminus\support{X_t}$. Furthermore we can assume that $f\pd{2}(t,x)$ is equal to $K$ whenever $x$ is an upper bound for $\support{X_t}$ and $k$ whenever it is a lower bound.

It only needs to be shown that $f(t,x)$ is decreasing in $t$. For any $s<t\in[0,T]$ we can apply Jensen's inequality,
\begin{equation*}
f(s,X_s) = \E{f(t,X_t)|\setsF_s}\ge f(t,X_s).
\end{equation*}
So $f(s,x)\ge f(t,x)$ for every $x$ in the support of $X_s$.
This inequality extends to the bounded open components of $\reals\setminus\support{X_s}$ as we chose $f(s,x)$ to be linear across these intervals. So, it holds for every $x\in[a,b]$ where $a$ is the infimum and $b$ is the supremum of $\support{X_s}$. If $b<\infty$ and $x>b$ the inequality $f\pd{2}(s,x)=K\ge f\pd{2}(t,x)$ gives
\begin{equation*}
f(s,x)=f(s,b)+K(x-b)\ge f(t,b)+K(x-b)\ge f(t,x).
\end{equation*}
Similarly, if $a>-\infty$ and $x<a$ the inequality $f\pd{2}(s,x)=k\le f\pd{2}(t,x)$ gives
\begin{equation*}
f(s,x)=f(s,a)+k(x-a)\ge f(t,a)+k(x-a)\ge f(t,x)
\end{equation*}
as required.
\end{proof}

\section{Marginal Supports}\label{sec:msupport}

In this section we shall prove a few results concerning the marginal support (Definition \ref{defn:marginal support}) of almost-continuous processes. In particular, we show that the paths of a process are contained in its marginal support (Lemma \ref{lemma:process lies in msupp}), and prove Lemma \ref{lemma:mgale AC iff never jumps past msupp} characterizing almost-continuous martingales in terms of their marginal support.

We start by showing that the marginal support is always Borel measurable.

\begin{lemma}\label{lemma:expression for msupp}
Let $X$ be a real valued stochastic process that is right-continuous in probability. For any real numbers $a<b$ set
\begin{equation*}
S_{a,b}=\left\{t\in\reals_+:(a,b)\cap\support{X_t}=\emptyset\right\}.
\end{equation*}
Then $S_{a,b}$ is Borel measurable and
\begin{equation}\label{eqn:lemma:expression for msupp}
\msupport{X}=\halfplane\setminus\bigcup\left\{S_{a,b}\times(a,b):a<b\in\rats\right\}.
\end{equation}
\end{lemma}
\begin{proof}
For any real numbers $a<b$ define $f:\reals_+\rightarrow\reals$ by
\begin{equation*}
f(t) = \E{\min\left((X_t-a)_+,(b-X_t)_+\right)}.
\end{equation*}
The right-continuity of $X$ implies that $f$ is also right-continuous, so it is Borel measurable.
Then,
\begin{equation*}
S_{a,b}=\left\{t\in\reals_+:f(t)=0\right\}
\end{equation*}
which is Borel measurable.

We now note that for every $t\ge 0$ the complement of the set $\support{X_t}$ is open, and therefore can be expressed as a union of open intervals $(a,b)$ for $a<b$. As the rational numbers are dense in the reals we can restrict to $a,b\in\rats$,
\begin{equation*}
\reals\setminus\support{X_t}=\bigcup\left\{(a,b): a<b\in\rats, t\in S_{a,b}\right\}.
\end{equation*}
So, for any $(t,x)\in\halfplane$ we get the following implications
\begin{eqnalign*}
(t,x)\not\in\msupport{X}&\iff&
x\not\in\support{X_t}\\
&\iff&x\in\bigcup\left\{(a,b): a<b\in\rats, t\in S_{a,b}\right\}\\
&\iff& (t,x)\in\bigcup\left\{S_{a,b}\times(a,b):a<b\in\rats\right\},
\end{eqnalign*}
which proves equality (\ref{eqn:lemma:expression for msupp}).
\end{proof}
The measurability of $\msupport{X}$ follows immediately.
\begin{lemma}
Let $X$ be a real valued stochastic process which is right-continuous in probability. Then its marginal support is a Borel measurable subset of $\halfplane$.
\end{lemma}
\begin{proof}
This follows immediately from equality (\ref{eqn:lemma:expression for msupp}).
\end{proof}

We now show that the paths of any right-continuous process are contained in its marginal support.

\begin{lemma}\label{lemma:process lies in msupp}
Let $X$ be a right-continuous real valued stochastic process. Then, with probability $1$, we have
\begin{equation*}
\left\{(t,X_t):t\in\reals_+\right\}\subseteq\msupport{X}.
\end{equation*}
\end{lemma}
\begin{proof}
We shall use proof by contradiction, so start by supposing that the statement is not true. Then, as the marginal support is Borel measurable there must exist a random time $\tau$ such that
\begin{equation*}
\Prob{(\tau,X_\tau)\not\in\msupport{X}}>0.
\end{equation*}
This follows from the Section Theorem (see \cite{Cohn} Corollary 8.5.4 or \cite{HeWangYan} Lemma 4.3).
Using the notation of Lemma \ref{lemma:expression for msupp}, it follows from equality (\ref{eqn:lemma:expression for msupp}) that there exists rational numbers $a<b$ such that
\begin{equation*}
\Prob{(\tau,X_\tau)\in S_{a,b}\times(a,b)} >0.
\end{equation*}
Without loss of generality, we can suppose that $(\tau,X_\tau)\in S_{a,b}\times(a,b)$ whenever $\tau<\infty$. We shall show that there are only countably many possible values that $\tau$ can take.

Define the random time
\begin{equation*}
\sigma = \inf\left\{t\in\reals_+: t\ge \tau, X_t\not\in(a,b)\right\}
\end{equation*}
and the set
\begin{equation*}
U = \left\{t\in\reals_+:\Prob{\tau<t<\sigma}>0\right\}.
\end{equation*}
For any $t\in U$ we have $\Prob{X_t\in(a,b)}>0$ so $(a,b)\cap\support{X_t}\not=\emptyset$. Therefore, $U$ and $S_{a,b}$ are disjoint. As $\tau\in S_{a,b}$ whenever $\tau<\infty$ this shows that
\begin{equation}
\Prob{\tau\in U} =0.\label{eqn:proof:process lies in msupp 1}
\end{equation}
Also, $U$ is open. To see this, choose any $t\in U$ and a sequence $(t_n)_{n\in\nat}$ such that $t_n\rightarrow t$. Then bounded convergence for expectations gives
\begin{equation*}
\liminf_{n\rightarrow\infty}\Prob{\tau<t_n<\sigma}\ge\Prob{\tau<t<\sigma}>0,
\end{equation*}
so $t_n\in U$ for large $n$. So we see that $U$ is indeed open. Therefore $U$ is a union of disjoint open intervals. We can write
\begin{equation*}
U = \bigcup_{n=1}^\infty (u_n,v_n)
\end{equation*}
where $u_n,v_n\in\reals_+\cup\{\infty\}$ and the intervals $(u_n,v_n)$ are disjoint.

Now, if $s<t$ are any times such that $\Prob{\tau<s<t<\sigma}>0$ then $(s,t)\subseteq U$. So, with probability one,
\begin{eqnalign*}
\srbrack \tau,\sigma\slbrack &=& \bigcup\left\{(s,t): s,t\in\rats_+, \tau<s<t<\sigma\right\}\\
&\subseteq& \bigcup\left\{(s,t):s,t\in\rats_+, \Prob{\tau<s<t<\sigma}>0\right\}\\
&\subseteq& U.
\end{eqnalign*}
Together with equation (\ref{eqn:proof:process lies in msupp 1}) this shows that whenever $\tau$ is finite then it is a left limit point of an interval in $U$ but is not in $U$, with probability one. So,
\begin{equation*}
\Prob{\tau=\infty\textrm{ or }\tau=u_n\textrm{ for some }n\in\nat}=1.
\end{equation*}
As promised, we have shown that there are only countably many possible values that $\tau$ can take. So there is a $t\in\reals_+$ such that
\begin{equation*}
\Prob{\tau=t}>0.
\end{equation*}
Finally this gives
\begin{eqnalign*}
0=\Prob{X_t\not\in\support{X_t}}\ge\Prob{\tau=t}>0,
\end{eqnalign*}
which is the required contradiction.
\end{proof}

In the case that $X$ is continuous in probability (but not necessarily with continuous paths) it is easy to extend this result to show that the left limits of the process are also in the marginal support.

\begin{lemma}\label{lemma:uniqueness:left lim of x is in msupp}
Let $X$ be a cadlag real valued stochastic process which is left-continuous in probability. Then, with probability $1$, we have
\begin{equation*}
\left\{(t,X_{t-}):t\in\reals_+\right\}\subseteq\msupport{X}.
\end{equation*}
\end{lemma}
\begin{proof}
We shall apply Lemma \ref{lemma:process lies in msupp} to this case simply by reversing the time direction. So pick any $n\in\nat$ and define the process
\begin{equation*}
Y_t = X_{(n-t)_+-}.
\end{equation*}
Lemma \ref{lemma:process lies in msupp} says that with probability one, for every $t\in\reals_+$ with $t\le n$,
\begin{equation*}
X_{t-}=Y_{n-t}\in\support{Y_{n-t}}=\support{X_{t-}}=\support{X_t},
\end{equation*}
as $X$ is left-continuous in probability. The result now follows by letting $n$ go to infinity.
\end{proof}

We are only interested in real valued processes here, but the above proofs would generalise quite easily to processes taking values in any Polish space.

We shall now move on to describe the `almost-continuous' property of one dimensional processes in terms of the marginal support. In particular, we shall prove Lemma \ref{lemma:mgale AC iff never jumps past msupp}.
First, we shall require the following useful, but very simple lemma.

\begin{lemma}\label{lemma:simple expectation at independent time}
Let $A$ be a bounded and jointly measurable stochastic process. Then defining
\begin{eqnalign*}
&&f:\reals_+\rightarrow\reals,\\
&&f(t) = \E{A_t},
\end{eqnalign*}
$f$ is Borel measurable. Furthermore if $\tau:\Omega\rightarrow\reals_+\cup\{\infty\}$ is a random time independent of $A$ then
\begin{equation*}
\E{1_{\{\tau<\infty\}}A_\tau}=\E{1_{\{\tau<\infty\}}f(\tau)}.
\end{equation*}
\end{lemma}
\begin{proof}
First choose any $A$-measurable set $S$ and time $s\in\reals_+$ and define the process $X$ by
\begin{equation*}
X_t=1_S1_{\{t\ge s\}}.
\end{equation*}
If we set
\begin{equation*}
g(t)=\E{X_t}=\E{1_S}1_{\{s\le t\}}
\end{equation*}
then $g$ is Borel measurable and the independence of $\tau$ and $S$ implies
\begin{equation*}
\E{1_{\{\tau<\infty\}}X_\tau}
=\E{\E{1_S}1_{\{s\le \tau<\infty\}}}
= \E{1_{\{\tau<\infty\}}g(\tau)}.
\end{equation*}
By the monotone class lemma this extends to any bounded $A$-measurable process $X$, and in particular applies to the case where $X=A$ and $g=f$.
\end{proof}

A simple corollary is that two independent processes that are continuous in probability cannot jump simultaneously.

\begin{corollary}\label{cor:no simultaneous jumps}
Let $Y,Z$ be cadlag real valued stochastic processes such that $Y$ is continuous in probability. Then,
\begin{equation*}
\Prob{\exists t\in\reals_+\textrm{ s.t. }Y_{t-}\not=Y_t\textrm{ and }Z_{t-}\not=Z_t}=0.
\end{equation*}
\end{corollary}
\begin{proof}
As $Z$ is cadlag, there exist $Z$-measurable random times $(\tau_n )_{n\in\nat}$ such that 
$\cup_{n\in\nat}\slbrack \tau_n\srbrack$ contains all the jump times of $Z$ almost-surely (see \cite{HeWangYan} Theorem 3.32). Then, Lemma \ref{lemma:simple expectation at independent time} with $A_t=1_{\{Y_t\not= Y_{t-}\}}$ gives
\begin{eqnalign*}
&&\Prob{\exists t\in\reals_+\textrm{ s.t. }Y_{t-}\not=Y_t\textrm{ and }Z_{t-}\not=Z_t}\\
&&\le\sum_{n\in\nat}\Prob{Y_{\tau_n-}\not=Y_{\tau_n}}
=\sum_{n\in\nat}\E{A_{\tau_n}}= 0.\ \ \ \ \ \ \qedhere
\end{eqnalign*}
\end{proof}

We can now prove the following lemma which gives a necessary and sufficient condition for the process $X$ to never jump past points in its marginal support. In particular, it implies that in the case where $X$ is almost-continuous then it cannot jump past points in its marginal support. However, the converse statement is not strong enough to say that $X$ will be almost continuous. If $Y$ and $Z$ are independent copies of $X$ then the second condition below says that $Y$ cannot jump from stricly below $Z$ to strictly above it. Unfortunately it does not rule out the possibility that $Y$ can approach $Z$ from below and then jump past it (so that $Y_{t-}=Z_{t-}=Z_t<Y_t$) which would contradict almost-continuity. In order to prove Lemma \ref{lemma:mgale AC iff never jumps past msupp} we will need to make use of the martingale property.

\begin{lemma}\label{lemma:never crosses MSupp cond}
If $X$ is a real valued process which is continuous in probability then the following are equivalent.
\begin{enumerate}
\item The set
\begin{equation*}
\left\{(t,x)\in\halfplane:X_{t-}<x<X_t\right\}.
\end{equation*}
is disjoint from $\msupport{X}$ with probability one.
\item Given two independent cadlag processes $Y$ and $Z$, each with the same distribution as $X$, then
\begin{equation*}
\Prob{\exists t\in\reals_+ {\rm\ s.t.\ } Y_{t-}<Z_t<Y_t}=0
\end{equation*}
\end{enumerate}
\end{lemma}
\begin{proof}
We first note that the second statement above is equivalent to stating that
\begin{equation}\label{eqn:pf:never cross MSupp cond:1}
\Prob{\exists t\in\reals_+ {\rm\ s.t.\ } Y_{t-}<a<Z_t<b<Y_t}=0
\end{equation}
for all real $a<b$. Let us fix any $a<b$.

As $Y$ is cadlag, there exist $Y$-measurable random times $(\tau_n)_{n\in\nat}$ such that $\cup_{n\in\nat}\slbrack \tau_n\srbrack$ contains all the jump times of $Y$ almost-surely (see \cite{HeWangYan} Theorem 3.32). So equation (\ref{eqn:pf:never cross MSupp cond:1}) is equivalent to saying that, for every $Y$-measurable random time $\tau$,
\begin{equation*}
\Prob{\tau<\infty{\rm\ and\ }Y_{\tau-}<a<Z_\tau<b<Y_\tau}=0.
\end{equation*}
Now define $f:\reals_+\rightarrow\reals$ by
\begin{equation*}
f(t) = \Prob{a<X_t<b}=\Prob{a<Z_t<b}.
\end{equation*}
Setting $A_t = 1_{\{a<Z_t<b\}}$ and
\begin{equation*}
\sigma=\left\{
\begin{array}{ll}
\tau,&\textrm{if $\tau<\infty$ and $Y_{\tau-}<a<b<Y_\tau$},\\
\infty,&\textrm{otherwise},
\end{array}
\right.
\end{equation*}
we can apply Lemma \ref{lemma:simple expectation at independent time},
\begin{eqnalign*}
\Prob{\tau<\infty{\rm\ and\ }Y_{\tau-}<a<Z_\tau<b<Y_\tau} &=& \E{1_{\{\sigma<\infty\}}A_\sigma}\\
&=&\E{1_{\{\sigma<\infty\}}f(\sigma)}.
\end{eqnalign*}
This term is zero if and only if $f(\sigma)=0$ whenever $\sigma<\infty$ (a.s.).
So, equation (\ref{eqn:pf:never cross MSupp cond:1}) is equivalent to saying that for every $Y$-measurable random time $\tau$ then,
\begin{equation*}
\Prob{\tau<\infty{\rm,\ }Y_{\tau-}<a<b<Y_\tau{\rm\ and\ } f(\tau)>0}=0.
\end{equation*}
However, as we noted above, the jump times of $Y$ are contained in $\cup_{n\in\nat}\slbrack \tau_n\srbrack$, so this is equivalent to saying that
\begin{eqnalign*}
\Prob{\exists t\in\reals_+{\rm\ s.t.\ }Y_{t-}<a<b<Y_t{\rm\ and\ }f(t)>0} &=&\\
\Prob{\exists t\in\reals_+{\rm\ s.t.\ }X_{t-}<a<b<X_t{\rm\ and\ }f(t)>0 } &=&0.
\end{eqnalign*}
Noting that $f(t)>0$ if and only if $(a,b)\cap\support{X_t}\not=\emptyset$, we see that equation (\ref{eqn:pf:never cross MSupp cond:1}) is therefore equivalent to
\begin{equation}\label{eqn:pf:never cross MSupp cond:2}
\Prob{\exists (t,x)\in\msupport{X}{\rm\ s.t.\ }X_{t-}<a<x<b<X_t}=0.
\end{equation}
However, saying that equation (\ref{eqn:pf:never cross MSupp cond:2}) is true for all real $a<b$ is equivalent to the first statement of the lemma, so the equivalence of equations (\ref{eqn:pf:never cross MSupp cond:1}) and (\ref{eqn:pf:never cross MSupp cond:2}) proves the result.
\end{proof}

This allows us to prove Lemma \ref{lemma:mgale AC iff never jumps past msupp} in one direction. Note that the result in this direction makes no use of the martingale property -- that will be necessary when we prove the converse.

\begin{corollary}\label{cor:AC implies never crosses MSupp}
Let $X$ be a real valued almost continuous process. Then the set
\begin{equation*}
\left\{(t,x)\in\halfplane:X_{t-}<x<X_t\textrm{ or }X_t<x<X_{t-}\right\}
\end{equation*}
is disjoint from the marginal support of $X$, with probability one.
\end{corollary}
\begin{proof}
Let $Y$ and $Z$ be independent cadlag processes with the same distribution as $X$. By Corollary \ref{cor:no simultaneous jumps}, the jump times of $Y$ and $Z$ are disjoint sets (restricting to a set of probability one).

Now suppose that, with positive probability, $Y_{t-}<Z_t<Y_t$ for some $t$. Then by the continuity of $Z$ at $t$ and the right-continuity of $Y$, there exist times $u,v\in\rats_+$ such that $u<t<v$ and $Y_u<Z_u$, $Y_v>Z_v$ and $Y_s\not=Z_s$ for every $s\in(u,v)$. As this contradicts Definition \ref{defn:acd} of almost-continuity, we see that the second statement of Lemma \ref{lemma:never crosses MSupp cond} is true.
Therefore, by Lemma \ref{lemma:never crosses MSupp cond}, the set
\begin{equation*}
\left\{(t,x)\in\halfplane:X_{t-}<x<X_t\right\}
\end{equation*}
is almost-surely disjoint from $\msupport{X}$. Similarly, applying the same argument to $-X$ gives the result.
\end{proof}

We still need to show that Lemma \ref{lemma:mgale AC iff never jumps past msupp} is true in the opposite direction, for which we need the following result. This is effectively saying that the process $Y$ cannot approach $Z$ from below without either touching or jumping past it.
Note that it also says that if the processes are adapted, then the stopping time $T$ is previsible on the set $Y_T\not=Z_T$.

\begin{lemma}\label{lemma:mgale never crosses msupp then ac:increasing times}
Let $X$ be a cadlag real valued process which is continuous in probability, and such that the set
\begin{equation*}
\left\{(t,x)\in\halfplane:X_{t-}<x<X_t\textrm{ or }X_t<x<X_{t-}\right\}
\end{equation*}
is disjoint from $\msupport{X}$ with probability one.

Also, let $Y$ and $Z$ be independent cadlag processes each with the same distribution as $X$.
For any $s\in\reals_+$ let $T$ be the random time
\begin{equation*}
T = \left\{
\begin{array}{ll}
\inf\{t\in\reals_+:t\ge s, Y_t\ge Z_t\},& \textrm{if }Y_s<Z_s,\\
\infty,&\textrm{otherwise},
\end{array}
\right.
\end{equation*}
and $(T_n)_{n\in\nat}$ be the random times
\begin{equation*}
T_n = \left\{
\begin{array}{ll}
\inf\{t\in\reals_+:t\ge s, Y_t+1/n\ge Z_t\},& \textrm{if }Y_s<Z_s,\\
\infty,&\textrm{otherwise}.
\end{array}
\right.
\end{equation*}
Then $T_n\uparrow T$ as $n\rightarrow\infty$ (a.s.). Also, $T_n<T$ whenever $T<\infty$ and $Y_T\not=Z_T$ (a.s.).
\end{lemma}
\begin{proof}
First, it is clear from the definitions that $T_n\le T$ for each $n$ and that $T_n$ is an increasing sequence. So we can define a random time $S$ by
\begin{equation*}
S = \lim_{n\rightarrow\infty}T_n.
\end{equation*}
We need to show that $S=T$, for which it is enough to prove $Y_S\ge Z_S$ whenever $S<\infty$.
We start by showing that $Y$ cannot have a negative jump at time $S$. Choose any $m\in\nat$ and let $A$ be the set
\begin{equation*}
A = \left\{\forall n\in\nat,\ T_n<S<\infty\textrm{ and } Y_S\le Y_{S-}-1/m\right\}.
\end{equation*}
Note that in $A$ we necessarily have $$|Z_{S-}-Y_{S-}|=\lim_{n\rightarrow\infty}|Z_{T_n}-Y_{T_n}|\le\lim_{n\rightarrow\infty}1/n=0,$$ so $Y_{S-}=Z_{S-}$.
Define the function
\begin{eqnalign*}
&&f:\halfplane\rightarrow\reals,\\
&&f(t,x) = \Prob{x-1/m<X_t<x},
\end{eqnalign*}
so $f(t,x)=0$ if and only if $(x-1/m,x)\cap\support{X_t}=\emptyset$.

The conditions of the lemma say that that $Y$ cannot jump past any points of the marginal support of $X$. However, on the set $A$, we have $(Y_{S-}-1/m,Y_{S-})\subseteq(Y_S,Y_{S-})$, so
\begin{equation*}
f(S,Z_{S-})=f(S,Y_{S-})=0\textrm{ (a.s.)}.
\end{equation*}
Restricting to $A$ we have $Y_{S-}=Z_{S-}$ and $Y_{t-}<Z_{t-}$ for $s< t<S$. So with probability one there exists a $u\in\rats_+$ with $u<S$ and such that
\begin{equation*}
Z_{t-}>Y_{t-}>Z_{t-}-1/m
\end{equation*}
for every $t\in[u,S)$.
As the paths of $Y_-$ lie in the marginal support of $X$ (by Lemma \ref{lemma:uniqueness:left lim of x is in msupp}) this implies that, restricting to $A$, $f(t,Z_{t-})>0$ for every $t\in[u,S)$ (almost surely).

So, if for every $u\in\rats_+$, we define the random time
\begin{equation*}
R^u=\inf\left\{t\in\reals_+:t\ge u, f(t,Z_{t-})=0\right\}
\end{equation*}
then $R^u=S$ for some $u\in\rats_+$. Note that the Debut Theorem (\citep{HeWangYan} Theorem 4.2) says that $R^u$ is $Z$-measurable.
Therefore,
\begin{equation*}
\Prob{A}\le\sum_{u\in\rats_+}\Prob{A\cap\{S=R^u\}}
\le\sum_{u\in\rats_+}\Prob{R^u<\infty,\ Y_{R^u-}\not=Y_{R^u}}.
\end{equation*}
As the random times $R^u$ are $Z$-measurable, they are independent of $Y$. So Lemma \ref{lemma:simple expectation at independent time} with $A_t=1_{\{Y_t\not=Y_{t-}\}}$ gives
\begin{equation*}
\Prob{R^u<\infty,\ Y_{R^u-}\not=Y_{R^u}}=0.
\end{equation*}
So $\Prob{A}=0$. More explicitly,
\begin{equation*}
\Prob{\forall n\in\nat,\ T_n<S<\infty\textrm{ and } Y_{S}\le Y_{S-}-1/m}=0
\end{equation*}
Letting $m$ go to infinity tells us that
\begin{equation*}
(\forall n\in\nat,\ T_n<S<\infty)\Rightarrow Y_S\ge Y_{S-}\textrm{ (a.s.)}.
\end{equation*}
Similarly, replacing $Y$ by $-Z$ and $Z$ by $-Y$ in the above argument gives
\begin{equation*}
(\forall n\in\nat,\ T_n<S<\infty)\Rightarrow Z_S\le Z_{S-}\textrm{ (a.s.)}.
\end{equation*}
However, we have $Y_{T_n}+1/n\ge Z_{T_n}$. Therefore, if $T_n<S<\infty$ for every $n$ then $Y_{S-}\ge Z_{S-}$. So, we get
\begin{equation*}
(\forall n\in\nat,\ T_n<S<\infty)\Rightarrow Y_S\ge Z_S\textrm{ (a.s.)}.
\end{equation*}
On the other hand, if $T_n=S<\infty$ for any $n$, then $$Y_S=Y_{T_m}\ge Z_{T_m}-1/m=Z_S-1/m$$ for every $m\ge n$, and therefore $Y_S\ge Z_S$.
We have shown that $Y_S\ge Z_S$ whenever $S<\infty$. So $T=S$ and $T_n\rightarrow T$.

Finally, suppose that $T_n=T<\infty$ for large $n$. Then, $Y_T\ge Z_T$ and $Y_{T-}\le Z_{T-}-1/n$.
Corollary \ref{cor:no simultaneous jumps} says that $Y_{T-}=Y_T$ or $Z_{T-}=Z_T$, so one of the following inequalities must be true,
\begin{equation*}
Z_{T}\le Y_T=Y_{T-}<Z_{T-}{\rm\ or\ }Y_{T-}<Z_{T-}=Z_T\le Y_T.
\end{equation*}
However Lemma \ref{lemma:never crosses MSupp cond} says that we cannot have $Y_{T-}<Z_T<Y_T$. Similarly, by replacing $Y$ with $-Z$ and $Z$ with $-Y$ then we cannot have $Z_T<Y_T<Z_{T-}$.
So, we must have $Y_T=Z_T$.
Therefore, if $Y_T\not=Z_T$ then $T_n<T$ for every $n$.
\end{proof}

Finally, we use the above result to prove Lemma \ref{lemma:mgale AC iff never jumps past msupp}.

\begin{proof}[Proof of Lemma \ref{lemma:mgale AC iff never jumps past msupp}]
First, if $X$ is almost-continuous then the result follows from Corollary \ref{cor:AC implies never crosses MSupp}. It only remains to show the converse.

So suppose that the set given by equation (\ref{eqn:lemma:mgale AC iff never jumps past msupp:1}) is almost surely disjoint from $\msupport{X}$.
Let $Y$ and $Z$ be independent cadlag processes each with the same distribution as $X$, and pick any $s\in\reals_+$.
Also, let $T$ and $(T_n)_{n\in\nat}$ be the stopping times defined by Lemma \ref{lemma:mgale never crosses msupp then ac:increasing times}. Picking any $t>s$ then the martingale property gives,
\begin{equation*}
\E{1_{\{T_n\le t\}}(Y_{t\wedge T}-Z_{t\wedge T})}=\E{1_{\{T_n\le t\}}(Y_{T_n}-Z_{T_n})}.
\end{equation*}
Noting that $Y_{T_n}<Z_{T_n}$ whenever $T_n<T$, and, by Lemma \ref{lemma:mgale never crosses msupp then ac:increasing times}, $Y_{T_n}=Z_{T_n}$ whenever $T_n=T<\infty$, we get
\begin{equation*}
\E{1_{\{T\le t\}}(Y_T-Z_T)}=\lim_{n\rightarrow\infty}\E{1_{\{T_n\le t\}}(Y_{T_n}-Z_{T_n})}\le 0.
\end{equation*}
As $Y_T\ge Z_T$ this shows that $Y_T=Z_T$ whenever $T\le t$. However, if $Y_s<Z_s$ and $Y_t>Z_t$ then $T<t$, so
\begin{equation*}
\Prob{Y_s<Z_s, Y_t>Z_t{\rm\ and\ }Y_u\not= Z_u\textrm{ for every }u\in(s,t)}=0.\qedhere
\end{equation*}
\end{proof}

\bibliography{expectations.bbl}
\bibliographystyle{plain}

\end{document}